\documentclass[11 pt]{article}
\usepackage{amsmath,amssymb,amsfonts,amsthm,graphicx,cite,epsfig,makecell,geometry, array,mathrsfs,cite}
\usepackage[varg]{pxfonts}
\usepackage{wrapfig}
\usepackage{enumitem}
\usepackage{epstopdf,footnote,multirow,booktabs,bigints,pstool,booktabs,caption,subcaption ,lipsum}
\usepackage[table]{xcolor}
\usepackage[colorlinks = true,
            linkcolor = red,
            citecolor = blue]{hyperref}
\usepackage[symbol]{footmisc}
\usepackage{titlesec, float}
\usepackage[export]{adjustbox}
\theoremstyle{plain}
\newtheorem{thm}{Theorem}[section]
\newtheorem{lm}[thm]{Lemma}

\theoremstyle{definition}
\newtheorem{de}[thm]{Definition}
\newtheorem{co}[thm]{Corollary}

\theoremstyle{remark}
\newtheorem{re}{\sc \textbf{Remark}}

\numberwithin{equation}{section}
\newcolumntype{?}{!{\vrule width 1.4pt}}

\titlelabel{\thetitle.\quad}
\renewenvironment{abstract}
               {\list{}{\rightmargin\leftmargin}%
                \item[\textbf{\hspace{8.6mm}Abstract ---}]\relax}
               {\endlist}
\DeclareUrlCommand{\url}{%
    \def\UrlLeft##1\UrlRight{\underline{##1}}}

\begin{document}
\title{Some estimation about Tayler-Maclaurin coefficients of generalized subclasses of bi-univalent functions }
\author{S.A. Saleh $^{\text{1}}$, Alaa H. El-Qadeem $^{\text{2}}$ and Mohamed A. Mamon $^{\text{3}}$
\and $^{1}${\small Dr.sasaleh@hotmail.com \ \& \ }$^{2}${\small  ahhassan@science.zu.edu.eg\ \& \ } $^{3}${\small mohamed.saleem@science.tanta.edu.eg}\\$^{\text{1,3}}${\small Department of Mathematics, Faculty of Science, Tanta University, Tanta 31527, Egypt} \\ $^{\text{2}}${\small Department of Mathematics, Faculty of Science, Zagazig University, Zagazig 44519, Egypt}}
\date{}
\maketitle

\begin{abstract}
Our objective in this paper is to introduce and investigate comprehensive-constructed subclasses of normalized analytic and bi-univalent functions on the unit open disc. Bounds for the second and third Tayler-Maclaurin coefficients of functions belonging to this subclasses were investigated. Furthermore, some improvement and connections to some of the previous known results are also pointed out.
\end{abstract}
%%%%%%%%%%%%%%%%%%%%%%%%%%%%%%%%%%%%%%%%%%%%%%%%%%%%%%%%%%%%%%%%%%%%%%%%%%%%%%%%%%%%%%%%%%%%%%%%%%%%%%%%%%
\section{Introduction, Definitions and Notations}
Let $\mathcal{A}$ denote the class of all analytic functions $f$ defined in the open unit disk $\mathbb{U}=\{z\in\mathbb{C}:\left\vert z\right\vert <1\}$ and normalized by the condition $f(0)= f^{\prime}(0)-1=0$. Thus each $f\in\mathcal{A}$ has a Taylor-Maclaurin series expansion of the form:
\begin{equation} \label{1.1}
f(z)=z+\sum\limits_{n=2}^{\infty}a_{n}z^{n}, \ \  (z \in\mathbb{U}).
\end{equation}
Further, let $\mathcal{S}$ denote the class of all functions $f \in\mathcal{A}$ which are univalent in $\mathbb{U}$ (for details, see \cite{Duren}; see also some of the recent investigations \cite{C5,C6}). Two of the important and well-investigated subclasses of the analytic and univalent function class $\mathcal{S}$ are the class
$\mathcal{S}^{\ast}(\alpha)$ of starlike functions of order $\alpha$ in $\mathbb{U}$ and the class $\mathcal{K}(\alpha)$ of convex functions of order $\alpha$ in $\mathbb{U}$. By definition, we have

\begin{equation}
\mathcal{S}^{\ast}(\alpha):=\left\{ f: \ f \in \mathcal{S} \ \ \text{and} \ \ \mbox{Re}\left\{ \frac{zf^{\prime }(z)}{f(z)}\right\} >\alpha,\quad (z\in \mathbb{%
U}; 0\leq \alpha <1) \right\},  \label{d1}
\end{equation}
and
\begin{equation}
\mathcal{K}(\alpha):=\left\{ f: \ f \in \mathcal{S} \ \ \text{and} \ \ \mbox{Re}\left\{ 1+\frac{zf^{\prime \prime }(z)}{f^{\prime }(z)}\right\} >\alpha,\quad (z\in \mathbb{%
U}; 0\leq \alpha <1) \right\}. \label{d2}
\end{equation}
It is clear from the definitions (\ref{d1}) and (\ref{d2}) that $\mathcal{K}(\alpha) \subset \mathcal{S}^{\ast}(\alpha)$. Also we have
\begin{equation*}
f(z) \in \mathcal{K}(\alpha) \ \ \text{iff} \ \ zf^{\prime}(z) \in \mathcal{S}^{\ast}(\alpha),
\end{equation*}
and
\begin{equation*}
f(z) \in \mathcal{S}^{\ast}(\alpha) \ \ \text{iff} \ \ \int_{0}^{z} \frac{f(t)}{t} dt =F(z) \in \mathcal{K}(\alpha).
\end{equation*}
It is well-known \cite{Duren} that every function $f\in \mathcal{S}$ has an inverse map $f^{-1}$ that satisfies the following conditions:
\begin{center}
$f^{-1}(f(z))=z \ \ \ (z\in \mathbb{U}),$
\end{center}
and
\begin{center}
$f\left(f^{-1}(w)\right)=w$ $\ \ \ \left( |w|<r_{0}(f);r_{0}(f)\geq\frac{1}{4}\right)$.
\end{center}
In fact, the inverse function is given by
\begin{equation} \label{1.2}
f^{-1}(w)=w-a_{2}w^{2}+(2a_{2}^{2}-a_{3})w^{3}-(5a_{2}^{3}-5a_{2}a_{3}+a_{4})w^{4}+\cdots.
\end{equation}

A function $f\in \mathcal{A}$ is said to be bi-univalent in $\mathbb{U}$ if both $f(z)$ and $f^{-1}(z)$ are univalent in $\mathbb{U}$. Let $\Sigma$ denote the class of bi-univalent functions in $\mathbb{U}$ given by (\ref{1.1}). For a brief history and some interesting examples of functions and characterization of the class $\Sigma$, see Srivastava et al. \cite{Sriv2010}, Frasin and Aouf \cite{Frasin}, and Magesh and Yamini \cite{F3}. Examples of functions in the class $\Sigma$ are
\begin{equation*}
\frac{z}{1-z}, \text{ \ \ } -log\left(1-z\right) \text{ \ \ and \ \ }  \frac{1}{2} log\left(\frac{1+z}{1-z}\right) .
\end{equation*}
and so on. However, the familiar Koebe function is not a member of $\Sigma$.
Other common examples of functions in $\mathcal{S}$ such as
\begin{equation*}
z-\frac{z^2}{2} \text{ \ \ and \ \ } \frac{z}{1-z^2}
\end{equation*}
are also not members of $\Sigma$.

In 1967, Lewin \cite{C21} investigated the bi-univalent function class $\Sigma $ and showed that $|a_{2}|<1.51$. Subsequently, Brannan and Clunie \cite{C22} conjectured that  $|a_{2}|\leq \sqrt{2}.$ Later, Netanyahu \cite{C23} showed that $\max$ $|a_{2}|=\frac{4}{3}$ if $f\in \Sigma.$ Brannan and Taha \cite{C25} introduced certain subclasses of a bi-univalent function class $\Sigma$ similar to the familiar subclasses $\mathcal{S}^{\ast}(\alpha)$ and $\mathcal{K}(\alpha)$ of starlike and convex functions of order $\alpha$ ($0\leq \alpha <1$), respectively (see \cite{C26}). Thus, following the works of Brannan and Taha \cite{C25}, for $0\leq \alpha <1,$ a function $f\in \Sigma $ is in the class $\mathcal{S}_{\Sigma}^{\ast }\left( \alpha \right) $ of bi-starlike functions of order $\alpha$;
or $\mathcal{K}_{\Sigma }\left( \alpha \right)$ of bi-convex functions of order $\alpha$ if both $f$ and $f^{-1}$ are respectively starlike or convex
functions of order $\alpha.$ Recently, many researchers have introduced and investigated several interesting subclasses of the bi-univalent function class $\Sigma$ and they have found non-sharp estimates on the first two Taylor-Maclaurin coefficients $|a_{2}|$ and $|a_{3}|$. In fact, the aforecited work of Srivastava et al. \cite{Sriv2010} essentially revived the investigation of various subclasses of the bi-univalent function class $\Sigma$ in recent years; it was followed by such works as those by Frasin and Aouf \cite{Frasin}, Xu et al. \cite{C210}, \c{C}a\u{g}lar et al. \cite{C29}, and others (see, for example, \cite{C212,F4}).
 The coefficient estimate problem for each of the following Taylor-Maclaurin coefficients $|a_{n}|$ $(n\in \mathbb{N}\backslash \{1,2\})$ for each $f\in \Sigma$  given by (\ref{1.1}) is still an open problem.

This work is concerned with the coefficient bounds for the Taylor-Maclaurin coefficients $|a_2|$ and $|a_3|$. Furthermore, we modify the definition of the classes introduced by Srivastava et la. \cite{S1} and modify estimations. Finally, several connections to some of the previous results are pointed out.
\begin{de}\label{def1}
For $0\leq\lambda\leq 1, 0\leq\delta\leq 1, 0\leq\mu\leq 1, 0\leq\gamma\leq 1$, $0<\alpha\leq1$ and $\tau\in \mathbb{C}^{\ast}=\mathbb{C}/\left\{0\right\}$. let $f\in \Sigma$ given by (\ref{1.1}), then $f$ is said to be in the class $\mathcal{S}_{\Sigma}^{\alpha}(\tau,\delta,\lambda,\gamma)$ if it satisfy the following conditions
\begin{equation*}
\left|\arg\left(1+\frac{1}{\tau}\left[\frac{(1-\delta)f(z)+\delta zf^{'}(z)+\mu z^2f^{''}(z)}{(1-\lambda)z+\lambda(1-\gamma)f(z)+\lambda\gamma zf^{'}(z)}-1\right]\right)\right|< \frac{\alpha\pi}{2},
\end{equation*}
and
\begin{equation*}
\left|\arg\left(1+\frac{1}{\tau}\left[\frac{(1-\delta)g(w)+\delta wg^{'}(w)+\mu w^2g^{''}(w)}{(1-\lambda)w+\lambda(1-\gamma)g(w)+\lambda\gamma wg^{'}(w)}-1\right]\right)\right|< \frac{\alpha\pi}{2},
\end{equation*}
for all $z\in \mathbb{U}$, $g=f^{-1}\in\Sigma$ given by (\ref{1.2}).
\end{de}
%%%%%%%%%%%%%%%%%%%%%%%%%%%%
\begin{de}\label{def2}
For $0\leq\lambda\leq 1, 0\leq\delta\leq 1, 0\leq\mu\leq 1, 0\leq\gamma\leq 1$, $0\leq\beta<1$ and $\tau\in \mathbb{C}^{\ast}=\mathbb{C}/\left\{0\right\}$. let $f\in \Sigma$ given by (\ref{1.1}), then $f$ is said to be in the class $\mathcal{S}_{\Sigma}(\tau,\delta,\mu,\lambda,\gamma;\beta)$ if it satisfy the following conditions
\begin{equation*}
\operatorname{Re}\left(1+\frac{1}{\tau}\left[\frac{(1-\delta)f(z)+\delta zf^{'}(z)+\mu z^2f^{''}(z)}{(1-\lambda)z+\lambda(1-\gamma)f(z)+\lambda\gamma zf^{'}(z)}-1\right]\right)>\beta,
\end{equation*}
and
\begin{equation*}
\operatorname{Re}\left(1+\frac{1}{\tau}\left[\frac{(1-\delta)g(w)+\delta wg^{'}(w)+\mu w^2g^{''}(w)}{(1-\lambda)w+\lambda(1-\gamma)g(w)+\lambda\gamma wg^{'}(w)}-1\right]\right)>\beta,
\end{equation*}
for all $z\in \mathbb{U}$, $g=f^{-1}\in\Sigma$ given by (\ref{1.2}).
\end{de}
%%%%%%%%%%%%%%%%%%%%%%%%%%%%%%
\begin{re}
If we put $\delta=1$ in the definitions \ref{def1} and \ref{def2}, we modified the definitions of the classes $\mathcal{H}_\Sigma(\tau,\mu,\lambda,\gamma;\alpha)$ and $\mathcal{H}_\Sigma(\tau,\mu,\lambda,\gamma;\beta)$ respectively which were introduced by Srivastava el at. \cite{S1}.
\end{re}
\begin{re}\
For special choices of the parameters $\tau,\delta,\mu,\lambda,\gamma$, we can
obtain the following subclasses as a special case of our main classes defined above:
\begin{enumerate}
  \item $\mathcal{S}_{\Sigma}^{\alpha}(1,\lambda,\delta,0, \gamma)=\mathcal{N}_{\Sigma}(\alpha,\lambda,\delta)$ which introduced by Serap Bulut \cite{Bulut}.
  \item $\mathcal{S}_{\Sigma}^{\alpha}(1,1,0,\lambda,0)=\mathcal{S}_{\Sigma}^{a,1,a}(\alpha,\lambda)$ and $\mathcal{S}_{\Sigma}(1,1,0,\lambda,0;\beta)=\mathcal{M}_{\Sigma}^{a,1,a}(\beta,\lambda)$ which were introduced by Srivastava et al. \cite{Magesh}.
  \item $\mathcal{S}_{\Sigma}^{\alpha}(1,1,0,1,\lambda)=\mathcal{G}_{\Sigma}(\alpha,\lambda)$ and $\mathcal{S}_{\Sigma}(1,1,0,1,\lambda;\beta)=\mathcal{M}_{\Sigma}(\beta,\lambda)$ which introduced by Murugusundaramoorthy et al. \cite{Prameela}.
  \item $\mathcal{S}_{\Sigma}^{\alpha}(1,1,\lambda,1,\lambda)=\mathfrak{B}_{\Sigma}(\alpha,\lambda)$ and $\mathcal{S}_{\Sigma}(1,1,\lambda,1,\lambda;\beta)=\mathcal{N}_{\Sigma}(\beta,\lambda)$ which introduced by Keerthi and Raja \cite{Raja}.
  \item $\mathcal{S}_{\Sigma}^{\alpha}(1,\lambda,0,0,\gamma)=\mathcal{B}_{\Sigma}(\alpha,\lambda)$ and $\mathcal{S}_{\Sigma}(1,\lambda,0,0,\gamma;\beta)=\mathcal{B}_{\Sigma}(\beta,\lambda)$ which introduced by Frasin and Aouf \cite{Frasin}.
  \item $\mathcal{S}_{\Sigma}^{\alpha}(1,1,\beta,0,\gamma)=\mathcal{H}_{\Sigma}(\alpha,\beta)$ and $\mathcal{S}_{\Sigma}(1,1,\beta,0,\gamma;\gamma)=\mathcal{H}_{\Sigma}(\gamma,\beta)$ which introduced by Frasin \cite{B.Frasin}.
  \item $\mathcal{S}_{\Sigma}^{\alpha}(1,1,0,0,\gamma)=\mathcal{H}_{\Sigma}^{\alpha}$ and $\mathcal{S}_{\Sigma}(1,1,0,0,\gamma;\beta)=\mathcal{H}_{\Sigma}(\beta)$ which introduced by Srivastava et al. \cite{Sriv2010}.
\end{enumerate}
\end{re}
%%%%%%%%%%%%%%%%%%%%%%%%%%%%%%%%%%
\begin{lm} \cite{Duren}
If $h\in \mathcal{P}$, then the estimates $|c_n|\leq 2$, $n = 1,2,3,...$ are sharp, where $\mathcal{P}$ is the family of all functions $h$ which are analytic in $\mathbb{U}$ for which $h(0) = 1$ and
$\operatorname{Re}(h(z))> 0 (z\in \mathbb{U})$ where
\[
h(z) = 1 + c_1 z + c_2 z^2+...,        z\in \mathbb{U}.
\]
\end{lm}
%%%%%%%%%%%%%%%%%%%%%%%%%%%%%%%%%%%%%%%%%%%%%%%%%%%%%%%%%%%%%%%%%%%%%%%%%%%%%%%
%%%%%%%%%%%%%%%%%%%%%%%%%%%%%%%%%%%%%%%%%%%%%%%%%%%%%%%%%%%%%%%%%%%%%%%%%%%
\section{Coefficient bounds for the function class $\mathcal{S}_{\Sigma}^{\alpha}(\tau,\delta,\lambda,\gamma)$}

In this section, we establish coefficient bounds for the Taylor-Maclaurin coefficients $|a_2|$ and $|a_3|$ of the function $f \in \mathcal{S}_{\Sigma}^{\alpha}(\tau,\delta,\lambda,\gamma)$.
\begin{thm}\label{thm1}
let $f(z)$ defined by (\ref{1.1}) belonging to the class $\mathcal{S}_{\Sigma}^{\alpha}(\tau,\delta,\lambda, \gamma)$, then
\begin{equation}\label{2.5}
    |a_{2}|\leq\frac{2\alpha|\tau|}{\sqrt{|2\alpha\tau\Omega+(1-\alpha)(1+\delta+2\mu-\lambda-\gamma\lambda)^2|}}.
\end{equation}
and
\begin{equation}\label{2.6}
|a_3|\leq \min\left\{\frac{4\alpha^2|\tau|^2}{(1+\delta+2\mu-\lambda-\gamma\lambda)^2}+
\frac{2\alpha|\tau|}{|1+2\delta+6\mu-\lambda-2\gamma\lambda|},\frac{2\alpha|\tau|}{|\Omega|}\right\}
\end{equation}
where
\begin{equation}\label{Omega}
\Omega=1+2\delta+6\mu-2\lambda-3\gamma\lambda-\lambda\delta-2\mu\lambda+\lambda^2+2\gamma\lambda^2-\delta\gamma\lambda-2\mu\gamma\lambda+\gamma^2\lambda^2.
\end{equation}
\end{thm}
%%%%%%%%%%%%%%%%%%%%%%%%%%%%
\begin{proof}
Since $f,g\in \mathcal{S}_{\Sigma}^{\alpha}(\tau,\delta,\lambda, \gamma)$, then there exist two functions $h_1(z)=1+\sum_{n=1}^{\infty}p_n z^n$ and $h_2(w)=1+\sum_{n=1}^{\infty}q_n w^n$ with positive real part in the unit disc such that
\begin{equation}\label{2.7}
1+\frac{1}{\tau}\left[\frac{(1-\delta)f(z)+\delta zf^{'}(z)+\mu z^2f^{''}(z)}{(1-\lambda)z+\lambda(1-\gamma)f(z)+\lambda\gamma zf^{'}(z)}-1\right]=(h_1(z))^\alpha
\end{equation}
and
\begin{equation}\label{2.8}
1+\frac{1}{\tau}\left[\frac{(1-\delta)g(w)+\delta wg^{'}(w)+\mu w^2g^{''}(w)}{(1-\lambda)w+\lambda(1-\gamma)g(w)+\lambda\gamma wg^{'}(w)}-1\right]=(h_2(w))^\alpha
\end{equation}
Now by comparing the coefficients of powers $z, z^2, w$ and $w^2$ in boss sides of equations (\ref{2.7}) and (\ref{2.8}), we obtain
\begin{equation}\label{2.9}
\frac{(1+\delta+2\mu-\lambda-\gamma\lambda)}{\tau}=\alpha p_1
\end{equation}
\begin{equation}\label{2.10}
\frac{(1+2\delta+6\mu-\lambda-2\gamma\lambda)a_3
-(1+\delta+2\mu-\lambda-\gamma\lambda)(\lambda+\gamma\lambda)a_{2}^{2}}{\tau}=\alpha p_2+\frac{\alpha(\alpha-1)}{2}p_1^2
\end{equation}
\begin{equation}\label{2.11}
-\frac{(1+\delta+2\mu-\lambda-\gamma\lambda)}{\tau}=\alpha q_1
\end{equation}
\begin{equation}\label{2.12}
\frac{(1+2\delta+6\mu-\lambda-2\gamma\lambda)(2a_{2}^{2}-a_3)
-(1+\delta+2\mu-\lambda-\gamma\lambda)(\lambda+\gamma\lambda)a_{2}^{2}}{\tau}=\alpha q_2+\frac{\alpha(\alpha-1)}{2}q_1^2.
\end{equation}
From equations (\ref{2.9}) and (\ref{2.11}), we deduce
\begin{equation}\label{2.13}
    p_1=-q_1,
\end{equation}
and
\begin{equation}\label{2.14}
    a_2=\frac{\alpha \tau p_1}{1+\delta+2\mu-\lambda-\gamma\lambda}.
\end{equation}
By adding equation (\ref{2.10}) to (\ref{2.12}), we obtain
\begin{equation}\label{2.15}
    2\frac{\Omega}{\tau}a_{2}^{2}=\alpha(p_2+q_2)+\frac{\alpha(\alpha-1)}{2}(p_1^{2}+q_1^{2})
\end{equation}
where $\Omega$ are given by (\ref{Omega}).\\
Using equation (\ref{2.13}) and substituting the value of $a_2$ from equation (\ref{2.14}) into equation (\ref{2.15}), we deduce
\begin{equation}\label{2.16}
    p_{1}^{2}=\frac{(p_2+q_2)(1+\delta+2\mu-\lambda-\gamma\lambda)^2}{2\alpha\tau\Omega+(1-\alpha)(1+\delta+2\mu-\lambda-\gamma\lambda)^2},
\end{equation}
therefore, by applying Lemma 1 in the equation (\ref{2.16}), then
\begin{equation}\label{2.17}
    |p_{1}|\leq\frac{2|1+\delta+2\mu-\lambda-\gamma\lambda|}{\sqrt{|2\alpha\tau\Omega+(1-\alpha)(1+\delta+2\mu-\lambda-\gamma\lambda)^2|}}.
\end{equation}
By substituting from equation (\ref{2.17}) into (\ref{2.14}) and using Lemma 1, we conclude the desired estimate of $a_2$ given by (\ref{2.5}).\\
On the other hand, to investigate the bounds of $|a_3|$, subtracting equation (\ref{2.12}) from (\ref{2.10}) we obtain
\begin{equation}\label{2.18}
2\frac{(1+2\delta+6\mu-\lambda-2\gamma\lambda)}{\tau}(a_3-a_{2}^{2})=\alpha(p_2-q_2)+\frac{\alpha(\alpha-1)}{2}(p_{1}^{2}-q_{1}^{2}),
\end{equation}
By using equation (\ref{2.13}) into (\ref{2.18}), we obtain
\begin{equation}\label{2.19}
a_3=a_{2}^{2}+\frac{\alpha\tau}{2(1+2\delta+6\mu-\lambda-2\gamma\lambda)}(p_2-q_2),
\end{equation}
Using Lemma 1 and substituting of the value of $a_2$ from (\ref{2.14}) into (\ref{2.19}), we conclude one of the desired estimates of $|a_3|$. \\
Now, by substituting the value of $a_2^2$ from equation (\ref{2.15}) into (\ref{2.19}), we obtain
\begin{equation}\label{2.199}
a_3=\frac{\alpha\tau}{2}\left[p_2\left(\frac{1}{\Omega}+\frac{1}{(1+2\delta+6\mu-\lambda-2\gamma\lambda)}\right)+q_2\left(\frac{1}{\Omega}-\frac{1}{(1+2\delta+6\mu-\lambda-2\gamma\lambda)}\right)\right],
\end{equation}
Applying Lemma 1 for the coefficients $p_2, q_2$, we conclude the other desired estimation of $|a_3|$.
\end{proof}
%%%%%%%%%%%%%%%%%%%%%%%%%%%%%%%%%%%%%%%%%%%%%%%%%%%%%%%%%%%%%%%%%%%%%%%%%%%%%%%%%%%%%%%%%%%%%%%%%%%%%%%%%%%%%%%%%%%%%%%%%%%%%
%%%%%%%%%%%%%%%%%%%%%%%%%%%%%%%%%%%%%%%%%%%%%%%%%%%%%%%%%%%%%%%%%%%%%%%%%%%%%%%%%%%%%%%%%%%%%%%%%%%%%%%%%%%%%%%%%%%%%%%%%%%%%
\section{Coefficient bounds for the function class $\mathcal{S}_{\Sigma}(\tau,\delta,\lambda,\gamma;\beta)$}

In this section, we establish coefficient bounds for the Maclaurin coefficients $|a_2|$ and $|a_3|$ of the function $f \in\mathcal{S}_{\Sigma}(\tau,\delta,\lambda,\gamma;\beta)$.
\begin{thm}\label{thm2}
let $f(z)$ defined by (\ref{1.1}) belonging to the class $\mathcal{S}_{\Sigma}(\tau,\delta,\lambda,\gamma;\beta)$, then
\begin{equation}\label{2.20}
    |a_{2}|\leq\min\left\{\frac{2(1-\beta)|\tau|}{|1+\delta+2\mu-\lambda-\gamma\lambda|},\sqrt{\frac{2|\tau|(1-\beta)}{|\Omega|}}\right\}.
\end{equation}
and
\begin{equation}\label{2.21}
|a_3| \leq \min\left\{\frac{4(1-\beta)^2|\tau|^2}{|1+\delta+2\mu-\lambda-\gamma\lambda|^2} +\frac{2(1-\beta)|\tau|}{|1+2\delta+6\mu-\lambda-2\gamma\lambda|},\frac{2(1-\beta)|\tau|}{|\Omega|}\right\}.
\end{equation}
\end{thm}
%%%%%%%%%%%%%%%%%%%%%%%%%%%%%%%%%%
\begin{proof}
Since $f,g\in \mathcal{S}_{\Sigma}(\tau,\delta,\lambda,\gamma;\beta)$, then there exist two functions $P(z)=\sum_{n=1}^{\infty}p_n z^n$ and $Q(w)=\sum_{n=1}^{\infty}q_n w^n$ with positive real part in the unit disc such that
\begin{equation}\label{2.22}
1+\frac{1}{\tau}\left[\frac{(1-\delta)f(z)+\delta zf^{'}(z)+\mu z^2f^{''}(z)}{(1-\lambda)z+\lambda(1-\gamma)f(z)+\lambda\gamma zf^{'}(z)}-1\right]=\beta+(1-\beta)P(z)
\end{equation}
and
\begin{equation}\label{2.23}
1+\frac{1}{\tau}\left[\frac{(1-\delta)g(w)+\delta wg^{'}(w)+\mu w^2g^{''}(w)}{(1-\lambda)w+\lambda(1-\gamma)g(w)+\lambda\gamma wg^{'}(w)}-1\right]=\beta+(1-\beta)Q(w)
\end{equation}
Now by comparing the coefficients of powers $z, z^2, w$ and $w^2$ in boss sides of equations (\ref{2.22}) and (\ref{2.23}), we obtain
\begin{equation}\label{2.24}
\frac{(1+\delta+2\mu-\lambda-\gamma\lambda)}{\tau}=(1-\beta) p_1
\end{equation}
\begin{equation}\label{2.25}
\frac{(1+2\delta+6\mu-\lambda-2\gamma\lambda)a_3
-(1+\delta+2\mu-\lambda-\gamma\lambda)(\lambda+\gamma\lambda)a_{2}^{2}}{\tau}=(1-\beta) p_2
\end{equation}
\begin{equation}\label{2.26}
-\frac{(1+\delta+2\mu-\lambda-\gamma\lambda)}{\tau}=(1-\beta) q_1
\end{equation}
\begin{equation}\label{2.27}
\frac{(1+2\delta+6\mu-\lambda-2\gamma\lambda)(2a_{2}^{2}-a_3)
-(1+\delta+2\mu-\lambda-\gamma\lambda)(\lambda+\gamma\lambda)a_{2}^{2}}{\tau}=(1-\beta) q_2.
\end{equation}
From equations (\ref{2.24}) and (\ref{2.26}), we deduce
\begin{equation}\label{2.28}
    p_1=-q_1,
\end{equation}
and
\begin{equation}\label{2.29}
    a_2=\frac{(1-\beta)\tau p_1}{1+\delta+2\mu-\lambda-\gamma\lambda}.
\end{equation}
By adding equation (\ref{2.25}) to (\ref{2.27}), we obtain
\begin{equation}\label{2.30}
    \frac{2\Omega a_2^2}{\tau}=(1-\beta)(p_2+q_2),
\end{equation}
where $\Omega$ are given by (\ref{Omega}).\\
Applying Lemma 1, we obtain
\begin{equation}\label{2.31}
|a_{2}|\leq\sqrt{\frac{2|\tau|(1-\beta)}{|\Omega|}}.
\end{equation}
On the other hand, from equation (\ref{2.29}) we can deduce also that
\begin{equation}\label{2.32}
|a_{2}|\leq\frac{2(1-\beta)|\tau|}{|1+\delta+2\mu-\lambda-\gamma\lambda|}.
\end{equation}
Combining this with inequality (\ref{2.31}), we obtain the desired estimate on the coefficient $|a_2|$
which given by (\ref{2.20}). \\
In order to deduce an estimation of bounds of $|a_3|$, subtracting equation
(\ref{2.27}) from (\ref{2.25}). we get
\begin{equation}\label{2.33}
a_3=a_{2}^{2}+\frac{\tau(1-\beta)(p_2-q_2)}{2(1+2\delta+6\mu-\lambda-2\gamma\lambda)},
\end{equation}
By substituting the value of $a_2$ from equation (\ref{2.29}) with applying Lemma 1 for the coefficients $p_1,p_2$ and $q_2$, we obtain
\begin{equation}\label{2.34}
|a_3|\leq \frac{4(1-\beta)^2 |\tau|^2}{(1+\delta+2\mu-\lambda-\gamma\lambda)^2}+\frac{2|\tau|(1-\beta)}{|1+2\delta+6\mu-\lambda-2\gamma\lambda|},
\end{equation}
Now, the value of $a_2$ from equation (\ref{2.30}) with applying Lemma 1 for the coefficients $p_2,q_2$, we obtain
 the other estimation of $|a_3|$ which given by (\ref{2.21}).
\end{proof}
%%%%%%%%%%%%%%%%%%%%%%%%%%%%%%%%%%%%%%%%%%%%%%%%%%%%%%%%%%%%%%%%%%%%%%%%%%%%%
\section{Some corollaries and consequences}

In this section, we introduced and modified some previous known results as an immediate consequences of Theorem \ref{thm1} and \ref{thm2}.\\
\begin{re}
Putting $\delta=1$, we modified the results considered by Srivastava et al. \cite[Theorem 1 and 2]{S1}.
\end{re}
\begin{co}
Let $f(z)$ given by (\ref{1.1}) be in the class $\mathcal{H}_\Sigma(\tau,\mu,\lambda,\gamma;\alpha)$ then,
\begin{equation*}\label{2.5}
    |a_{2}|\leq\frac{2\alpha|\tau|}{\sqrt{|2\alpha\tau\tilde{\Omega}+(1-\alpha)(2+2\mu-\lambda-\gamma\lambda)^2|}}.
\end{equation*}
and
\begin{equation*}\label{2.6}
|a_3|\leq \min\left\{\frac{4\alpha^2|\tau|^2}{(2+2\mu-\lambda-\gamma\lambda)^2}+
\frac{2\alpha|\tau|}{|3+6\mu-\lambda-2\gamma\lambda|},\frac{2\alpha|\tau|}{|\tilde{\Omega}|}\right\}
\end{equation*}
where
\[
\tilde{\Omega}=\gamma^2\lambda^2+2\gamma\lambda^2-2\gamma\lambda\mu-4\gamma\lambda+\lambda^2-2\lambda\mu-3\lambda+6\mu+3
\]
\end{co}
\begin{co}
Let $f(z)$ given by (\ref{1.1}) be in the class $\mathcal{H}_\Sigma(\tau,\mu,\lambda,\gamma;\beta)$ then,
\begin{equation*}\label{2.20}
    |a_{2}|\leq\min\left\{\frac{2(1-\beta)|\tau|}{|2+2\mu-\lambda-\gamma\lambda|},\sqrt{\frac{2|\tau|(1-\beta)}{|\tilde{\Omega}|}}\right\}.
\end{equation*}
and
\begin{equation*}\label{2.21}
|a_3| \leq \min\left\{\frac{4(1-\beta)^2|\tau|^2}{|2+2\mu-\lambda-\gamma\lambda|^2} +\frac{2(1-\beta)|\tau|}{|3+6\mu-\lambda-2\gamma\lambda|},\frac{2(1-\beta)|\tau|}{|\tilde{\Omega}|}\right\}.
\end{equation*}
\end{co}
%%%%%%%%%%%%%%%%%%%%%%%%%%%
\begin{re}
Putting $\tau=1, \lambda=0$ and replace $\delta$ by $\lambda$ and also $\mu$ by $\delta$, we obtain the result considered by Bulut \cite[Theorem 5]{Bulut}.
\end{re}
\begin{co}
Let $f(z)$ given by (\ref{1.1}) be in the class $\mathcal{N}_\Sigma(\beta,\lambda,\delta)$ then,
\begin{equation*}
|a_2|\leq \min\left\{\frac{2(1-\beta)}{1+\lambda+2\delta},\sqrt{\frac{2(1-\beta)}{1+2\lambda+6\delta}}\right\},
\end{equation*}
\begin{equation*}
|a_3|\leq \frac{2(1-\beta)}{1+2\lambda+6\delta}.
\end{equation*}
\end{co}
%%%%%%%%%%%%%%%%%%%%%%%%%%%%
\begin{re}
Putting $\tau=\delta=\lambda=1, \mu=0$ and replace $\gamma$ by $\lambda$, we obtain the result considered by Murugusundaramoorthy et al. \cite[Theorem 4 and 5]{Prameela}.
\end{re}
\begin{co}
Let $f(z)$ given by (\ref{1.1}) be in the class $\mathcal{G}_\Sigma(\alpha,\lambda)$ then,
\begin{equation*}
|a_2|\leq \frac{2\alpha}{(1-\lambda)\sqrt{1+\alpha}},
\end{equation*}
\begin{equation*}
|a_3|\leq \min\left\{\frac{4\alpha^2}{(1-\lambda)^2}+\frac{\alpha}{1-\lambda},\frac{2\alpha}{(1-\lambda)^2}\right\}.
\end{equation*}
\end{co}
\begin{co}
Let $f(z)$ given by (\ref{1.1}) be in the class $\mathcal{M}_\Sigma(\beta,\lambda)$ then,
\begin{equation*}
|a_2|\leq \frac{\sqrt{2(1-\beta)}}{1-\lambda},
\end{equation*}
\begin{equation*}
|a_3|\leq \min\left\{\frac{2(1-\beta)}{(1-\lambda)^2},\frac{4(1-\beta)^2}{(1-\lambda)^2}+\frac{1-\beta}{1-\lambda}\right\}.
\end{equation*}
\end{co}
%%%%%%%%%%%%%%%%%%%%%%%%%%%
\begin{re}
Putting $\tau=\delta=1$ and $\mu=\gamma=0$, we obtain the result considered by Srivastava et al. \cite[Theorem 2.1 and 3.1 with $a=c,b=1$]{Magesh}.
\end{re}
\begin{co}
Let $f(z)$ given by (\ref{1.1}) be in the class $\mathcal{S}_\Sigma^{a,1,a}(\alpha,\lambda)$ then,
\begin{equation*}
|a_2|\leq \frac{2\alpha}{\sqrt{|2\alpha(\lambda^2-3\lambda+3)+(1-\alpha)(2-\lambda)^2|}},
\end{equation*}
\begin{equation*}
|a_3|\leq \min\left\{\frac{4\alpha^2}{(2-\lambda)^2}+\frac{2\alpha}{3-\lambda},\frac{2\alpha}{\lambda^2-3\lambda+3}\right\}.
\end{equation*}
\end{co}
\begin{co}
Let $f(z)$ given by (\ref{1.1}) be in the class $\mathcal{M}_\Sigma^{a,1,a}(\beta,\lambda)$ then,
\begin{equation*}
|a_2|\leq \sqrt{\frac{2(1-\beta)}{\lambda^2-3\lambda+3}},
\end{equation*}
\begin{equation*}
|a_3|\leq \min\left\{\frac{2(1-\beta)}{\lambda^2-3\lambda+3},\frac{4(1-\beta)^2}{(2-\lambda)^2}+\frac{2(1-\beta)}{3-\lambda}\right\}.
\end{equation*}
\end{co}
%%%%%%%%%%%%%%%%%%%%%%%%%%%%
\begin{re}
Putting $\tau=\delta=\lambda=1$ and replace $\mu,\gamma$ by $\lambda$, we obtain the result considered by Keerthi and Raja \cite[Corollary 2.3 and 3.4]{Raja}.
\end{re}
\begin{co}
Let $f(z)$ given by (\ref{1.1}) be in the class $\mathfrak{B}_\Sigma(\alpha,\lambda)$ then,
\begin{equation*}
|a_2|\leq \frac{2\alpha}{\sqrt{\left|4\alpha(1+2\lambda)+(1-3\alpha)(1+\lambda)^2\right|}},
\end{equation*}
\begin{equation*}
|a_3|\leq \min\left\{\frac{4\alpha^2}{(1+\lambda)^2}+\frac{\alpha}{1+2\lambda},\frac{2\alpha}{1+2\lambda-\lambda^2}\right\}.
\end{equation*}
\end{co}
\begin{co}
Let $f(z)$ given by (\ref{1.1}) be in the class $\mathcal{N}_\Sigma(\beta,\lambda)$ then,
\begin{equation*}
|a_2|\leq \sqrt{\frac{2(1-\beta)}{1+2\lambda-\lambda^2}},
\end{equation*}
\begin{equation*}
|a_3|\leq \min\left\{\frac{2(1-\beta)}{1+2\lambda-\lambda^2},\frac{4(1-\beta)^2}{(1+\lambda)^2}+\frac{1-\beta}{1+2\lambda}\right\}.
\end{equation*}
\end{co}
%%%%%%%%%%%%%%%%%%%%%%%%%%%%%%%
\begin{re}
Putting $\tau=1,\lambda=\mu=0$ and replace $\delta$ by $\lambda$, we obtain the result considered by Frasin and Aouf \cite[Theorem 2.2 and 3.2]{Frasin}.
\end{re}
\begin{co}\label{co1}
Let $f(z)$ given by (\ref{1.1}) be in the class $\mathcal{B}_\Sigma(\alpha,\lambda)$ then,
\begin{equation*}
|a_2|\leq \frac{2\alpha}{\sqrt{|\alpha(1+2\lambda-\lambda^2)+(1+\lambda)^2}},
\end{equation*}
\begin{equation*}
|a_3|\leq \frac{2\alpha}{1+2\lambda}.
\end{equation*}
\end{co}
\begin{co}\label{co4}
Let $f(z)$ given by (\ref{1.1}) be in the class $\mathcal{B}_\Sigma(\beta,\lambda)$, then
\begin{equation*}
|a_2|\leq \sqrt{\frac{2(1-\beta)}{1+2\lambda}},
\end{equation*}
\begin{equation*}
|a_3| \leq \frac{2(1-\beta)}{1+2\lambda}.
\end{equation*}
\end{co}
%%%%%%%%%%%%%%%%%%%%%%%%%%%%%%
\begin{re}
Putting $\tau=\delta=1,\lambda=0$ and replace $\mu$ by $\beta$, we obtain the result considered by Frasin \cite[Theorem 2.2 and 3.2]{B.Frasin}.
\end{re}
\begin{co}\label{co2}
Let $f(z)$ given by (\ref{1.1}) be in the class $\mathcal{H}_\Sigma(\alpha,\beta)$ then,
\begin{equation*}
|a_2|\leq \frac{2\alpha}{\sqrt{|2(2+\alpha)+4\beta(\alpha+\beta+2-\alpha\beta)|}},
\end{equation*}
\begin{equation*}
|a_3|\leq \frac{2\alpha}{3(1+2\beta)}.
\end{equation*}
\end{co}
\begin{co}\label{co5}
Let $f(z)$ given by (\ref{1.1}) be in the class $\mathcal{H}_\Sigma(\gamma,\beta)$, then
\begin{equation*}
|a_2|\leq \sqrt{\frac{2(1-\gamma)}{3(1+2\gamma)}},
\end{equation*}
\begin{equation*}
    |a_3| \leq \frac{2(1-\gamma)}{3(1+2\beta)}.
\end{equation*}
\end{co}
%%%%%%%%%%%%%%%%%%%%%%%%%%%%%%%%%%
\begin{re}
Putting $\tau=\delta=1$ and $\mu=\lambda=0$, we obtain the result considered by Srivastava et al. \cite[Theorem 1 and 2]{Sriv2010}.
\end{re}
\begin{co}\label{co3}
Let $f(z)$ given by (\ref{1.1}) be in the class $\mathcal{H}_\Sigma^{\alpha}$ then,
\begin{equation*}
|a_2|\leq \alpha\sqrt{\frac{2}{2+\alpha}},
\end{equation*}
\begin{equation*}
|a_3|\leq \frac{2\alpha}{3}.
\end{equation*}
\end{co}
\begin{co}\label{co6}
Let $f(z)$ given by (\ref{1.1}) be in the class $\mathcal{H}_\Sigma(\beta)$, then
\begin{equation*}
|a_2|\leq \sqrt{\frac{2(1-\beta)}{3}},
\end{equation*}
\begin{equation*}
    |a_3| \leq  \frac{2(1-\beta)}{3}.
\end{equation*}
\end{co}
%%%%%%%%%%%%%%%%%%%%%%%%%%%%%%%%
\begin{re}
Finally, we introduce some results improved in our paper,
\begin{itemize}
  \item [1.] The estimation of $|a_3|$ in Corollaries \ref{co1}, \ref{co4} are improvement of the estimates obtained by Frasin and Aouf \cite[Theorem 2.2 and 3.2]{Frasin}.
  \item [2.] The estimation of $|a_3|$ in Corollaries \ref{co2}, \ref{co5} are improvement of the estimates obtained by Frasin \cite[Theorem 2.2 and 3.2]{B.Frasin}.
  \item [3.] The estimation of $|a_3|$ in Corollaries \ref{co3}, \ref{co6} are improvement of the estimates obtained by Srivastava et al. \cite[Theorem 1 and 2]{Sriv2010}.
\end{itemize}
\end{re}

\end{document}